\documentclass[10pt]{amsart}

\usepackage{amsfonts,amsmath,amssymb}

\usepackage{enumerate}

\numberwithin{equation}{section}

\newtheorem{theorem}{Theorem}[section]

\newtheorem{lemma}[theorem]{Lemma}

\newtheorem{proposition}[theorem]{Proposition}

\newtheorem{examples}[theorem]{Examples}

\newtheorem{definition}{Definition}[section]

\newtheorem{corollary}[theorem]{Corollary}

\newcommand{\cl}[1]{\mathcal{#1}} %\cl A

\newcommand{\bb}[1]{\mathbb{#1}}

\newcommand{\Map}[1]{\mathrm{Map}(#1)} %\Map (U)

\newcommand{\Lat}[1]{\mathrm{Lat}(#1)} %\Lat{A}

\newcommand{\Alg}[1]{\mathrm{Alg}(#1)} %\Alg{L} 

\begin{document}

\title{Bilattices and Morita equivalence of masa bimodules}

\author[G.~K.~Eleftherakis]{G. K. Eleftherakis}

\address{G. K. Eleftherakis\\ University of Patras\\Faculty of Sciences\\
Department of Mathematics\\265 00 Patras Greece }

\email{gelefth@math.upatras.gr}

%\keywords{TRO, stable isomorphism}


\date{}

\maketitle

\begin{abstract}We define an equivalence relation between bimodules over maximal abelian 
selfadjoint algebras (masa bimodules) which we call 
spatial Morita equivalence. We prove that two reflexive masa bimodules are spatially Morita 
equivalent 
iff their (essential)  bilattices are isomorphic. We also prove that if $\cl S^1, \cl S^2$ are 
bilattices which correspond to reflexive masa bimodules $\cl U_1, \cl U_2$ and $f: \cl S^1\rightarrow 
\cl S^2$ is an onto bilattice homomorphism, then:\\(i) If $\cl U_1$ is synthetic, 
then $\cl U_2$ is synthetic. \\ (ii) If $\cl U_2$ contains a nonzero compact 
(or a finite or a rank 1) operator, then 
$\cl U_1$ also contains a nonzero compact (or a finite or a rank 1) operator. 
\end{abstract}

\section{Introduction}

If $H_1, H_2$ are Hilbert spaces, $B(H_1, H_2)$ is the space of bounded operators from $H_1$ into 
$H_2.$ When $H_1=H_2=H$ we write $B(H)=B(H_1, H_2).$ In this 
case $\cl P(B(H))$ is the set of orthogonal projections of $H.$ If $\cl U$ is a subset of $B(H_1, H_2)$ 
then its reflexive hull, \cite{ls}, is
$$\mathrm{Ref}(\cl U)=\{T\in B(H_1, H_2): T(\xi )\in \overline{\mbox{span}(\cl U\xi) }\;\;\forall \;\xi \;\in \;
H_1 \}. $$ If $\cl L$ is a lattice of orthogonal projections of a Hilbert space $H$, then 
the corresponding algebra $$ \mathrm{Alg}(\cl L) =\{T\in B(H): TL(H)\subset L(H)\;\;\forall \;L\;\in 
\; \cl L\}$$ is reflexive in the sense of the above definition.

An important class of non-selfadjoint operator algebras are the CSL algebras introduced by Arveson
 \cite{arv}. A CSL (commutative  subspace lattice) $\cl L$ is a commuting set of orthogonal projections acting on a Hilbert space 
which contains the zero and identity operators.  We also assume that $\cl L$ is closed in arbitrary suprema 
and infima. The corresponding CSL algebra is $\mathrm{Alg}(\cl L).$ 
In the sequel, a maximal abelian selfadjoint algebra is called masa. If $\cl D_i$ is a masa 
acting on the Hilbert space $H_i, i=1,2$ and $\cl U$ is a subspace of 
$B(H_1, H_2)$ such that $\cl D_2\cl U\cl D_1\subset \cl U$ we call $\cl U$ masa bimodule. It is not difficult to 
see that a reflexive algebra is a CSL algebra iff it contains a  
masa. Therefore the reflexive masa bimodules are generalization of CSL algebras.

If $\cl A$ is an algebra acting on a Hilbert space $H,$ then the set 
$$ \mathrm{Lat} (\cl A)=\{P\in \cl P(B(H)): P^\bot \cl A P=0\}$$ is a lattice. A lattice $\cl L$  
is called reflexive if $\cl L=\mathrm{Lat} (\mathrm{Alg}(\cl L) ).$ The CSLs are reflexive 
lattices, \cite{arv}. 

The interaction between reflexive algebras $\cl A$ and lattices  $\mathrm{Lat} (\cl A)$ is very 
interesting. The correspondence between CSL algebras and CSL lattices has been generalized 
to the case of masa bimodules by Shulman and Turowska, \cite{st}. 
In their paper, to every masa bimodule there corresponds 
at least one bilattice. 

In the present paper, we study bilattice homomorphisms between the bilattices of reflexive masa bimodules. 
We prove that if $\cl S^1, \cl S^2$ are isomorphic strongly closed commutative bilattices, then the corresponding 
masa bimodules $\cl U_1, \cl U_2$ are spatially Morita equivalent in the sense that the one is generated by the 
other, see Theorem \ref{31}. Conversely, if $\cl U_1, \cl U_2$ are spatially Morita equivalent, 
then their corresponding (essential) bilattices are isomorphic.

If $\cl U$ is a weak* closed masa bimodule, then there exists \cite{arv, st} a smallest  
weak* closed masa bimodule $ \cl U_{min} $ containing in $\cl U$ such that $$ \mathrm{Ref}(\cl U_{min
})=\mathrm{Ref}(\cl U).$$   In case $\cl U=\cl U_{min} $, we call $\cl U$  a synthetic 
masa bimodule. The concept of a synthetic bimodule is important in operator theory. In Section 4 
we will prove that if $\cl S^1, \cl S^2$ are strongly closed commutative bilattices, $\cl U_1, \cl U_2$ 
are their corresponding masa bimodules,  and there exists an onto bilattice homomorphism 
$f: \cl S^1\rightarrow \cl S^2$, then $\cl U_1$'s being synthetic implies that 
$\cl U_2$ is synthetic. In this case, if $\cl U_2$ contains a nonzero compact (finite rank, rank 1) operator,  
then $\cl U_1$ also  contains a nonzero compact (finite rank, rank 1) operator. We stress 
that the problem to characterise all CSLs such that the corresponding CSL algebras contain a nonzero 
compact operator is still open. 

In the last section we will give a new proof of the inverse image theorem, \cite{st}, using 
the results of this paper. It is important that our proof gives 
more information  about the sets satisfying the assumptions of the  inverse image theorem:
If $E$ is the preimage of the Borel set $E_1$ with respect to the appropriate Borel mappings and there
 exists a  nonzero compact (or a finite rank or a rank 1) operator supported on $E$, then 
there exists a  nonzero compact (or a finite rank or a rank 1) operator supported on $E_1,$ Theorem \ref{54}.

 If $\cl A$ is a von Neumann algebra, we denote by $\cl P(\cl A)$ the subset of its projections. Let $\cl U\subset B(H_1, H_2)$ be 
a subspace. We denote by $\Map {\cl U}$ the map sending every projection $P\in \cl P(B(H_1))$ to the projection onto the subspace of 
$H_2$ generated by the vectors of the form $\{UP\xi : U\in \cl U, \xi \in H_1\}.$ This map is sup preserving. Suppose that 
$\phi =\Map {\cl U}.$ We denote by $\phi ^*$ the $\Map {\cl U^*}.$ We also write 
$$\cl S_{2,\phi }=\{\phi (P ): P\in \cl P(B(H_1))\}, \;\;\;\cl S_{1,\phi }=\cl S_{2, \phi ^*}^\bot .$$ J. A. 
 Erdos has proved, \cite{erd},
 that $ \cl S_{1,\phi } $ is $\wedge$-complete, $ \cl S_{2, \phi } $ is $\vee $-complete $\phi (\cl S_{1,\phi } )
=\cl S_{2, \phi } ,$  the restriction of $\phi $ to $\cl S_{1,\phi } $ is $1-1$, and if $\cl U$ is a reflexive space, then 
$$\cl U=\{T\in B(H_1, H_2): \phi (P)^\bot TP=0\;\;\forall \;\;P\;\in \;\cl S_{1,\phi }\}.$$ If $\cl U$ is a masa bimodule, then 
the elements of $\cl S_{1,\phi }$ and  $\cl S_{2,\phi }$  commute.

A  set $\cl S\subset \cl P(B(H_1))\times \cl P(B(H_2))$ is called a bilattice if $(0,0), (I, 0), (0,I)\in \cl S$ and whenever 
$(P_1, Q_1), (P_2, Q_2)\in \cl S$, then $(P_1\wedge P_2, Q_1\vee Q_2), (P_1\vee P_2, Q_1\wedge Q_2)\in \cl S.$ We use the following notation. 
$$ \cl S_l=\{P: (P,Q)\in \cl S \}, \cl S_r=\{Q: (P,Q)\in \cl S \}   $$
$$\mathfrak M (\cl S)=\{T\in B(H_1, H_2): QTP=0\;\;\forall \;\;(P,Q)\in \cl S\}.$$ 
The space $\mathfrak M (\cl S)$ is reflexive in the usual sense. If the elements of 
 $\cl S_l$ and $ \cl S_r$ commute, then $\mathfrak M (\cl S)$ 
is a masa bimodule. If $\cl U\subset B(H_1, H_2)$ is a space, we denote by $Bil(\cl U)$ the following set:
$$\{(P,Q)\in \cl P(B(H_1))\times \cl P(B(H_2)): Q\cl UP=0\}.$$ Clearly  $Bil(\cl U)$ is a bilattice.

If $\cl L_1, \cl L_2$ are CSLs acting on $H_1, H_2$, respectively, we say that a map 
$\theta : \cl L_1\rightarrow \cl L_2$ is 
a CSL homorphism if $\theta $ preserves arbitrary suprema  and infima. 
 If additionally $\theta $ is $1-1$ and onto, we say that 
it is a CSL isomorphism. In this case, we denote by $Op(\theta )$ the space
$$\{T\in B(H_1, H_2): \theta (P)^\bot TP=0\;\;\forall \;\;P\;\;\in \cl L_1\}.$$

\medskip

In what follows, if $\cl X$ is a set of operators, we denote by $[\cl X]$ the weak* closure of the linear span of $\cl X.$ We 
recall the following:

\begin{definition}\label{11} \cite{ele} Let $\cl A_i$ be a weak* closed algebra acting on $H_i, i=1,2.$ We call 
$\cl A_1$ and $ \cl A_2$ 
spatially Morita equivalent if there exist an $\cl A_1-\cl A_2$ module $\cl X$ and an $\cl A_2-\cl A_1$ module $\cl Y$ such that 
$$\cl A_1=[\cl X \cl Y], \;\;\cl A_2=[\cl Y \cl X].$$
\end{definition}

\begin{theorem}\label{12}\cite{ele} Let $\cl A_1= \mathrm{Alg}(\cl L_1) , \cl A_2=\mathrm{Alg}(\cl L_2) $ be CSL algebras. Then $\cl A_1$ and $\cl A_2$ are spatially Morita equivalent if and only if 
there exists a CSL isomorphism $\theta : \cl L_1\rightarrow \cl L_2.$ In this case 
$$\cl A_1=[Op(\theta ^{-1})Op(\theta )],\;\;\; \cl A_2=[Op(\theta )Op(\theta ^{-1})].$$
\end{theorem}

We call a space $\cl U\subset B(H_1, H_2)$ nondegenerate if $\cl U(H_1)$ (resp. $\cl U^*(H_2)$) is dense in $H_2$ (resp. in $H_1$).

We extend the notion of spatial Morita equivalence to the setting of spaces:

\begin{definition}\label{13} We fix masas acting on the Hilbert spaces $H_i, K_i, i=1,2.$  
Let $\cl U_1\subset B(H_1, H_2), \cl U_2\subset B(K_1, K_2)$ be nondegenerate bimodules over the above masas. We call them 
spatially Morita equivalent if there exist bimodules over  the same masas   
$$ \cl V_1\subset B(H_1, K_1), \;\;\cl V_2\subset B(H_2, K_2),\;\;\cl W_1\subset B(K_1, H_1) ,\;\;\cl W_2\subset B(K_2, H_2) $$
 such that 

$(i)$ $$ \cl U_1=[\cl W_2\cl U_2\cl V_1], \;\;\cl U_2=[\cl V_2\cl U_1\cl W_1], $$ and 

$(ii)$ the spaces $$[\cl V_1 \cl W_1], \;\;[\cl W_1\cl V_1], \;\;[\cl V_2\cl W_2],\;\;[\cl W_2\cl V_2]$$ 
are unital algebras.
\end{definition}

\begin{examples}\em{
$(i)$ We can easily prove that if two unital algebras containing masas are spatially Morita equivalent as algebras, then they are spatially Morita equivalent 
as spaces.

$(ii)$ We call the spaces $\cl U_1, \cl U_2$ unitarily equivalent if there exist unitaries $W, V$ such that $\cl U_1=W\cl U_2V.$ Clearly 
unitarily equivalent masa bimodules are spatially Morita equivalent.

$(iii)$ In \cite{ept}, two spaces $\cl U_1\subset B(H_1, H_2), \cl U_2\subset B(K_1, K_2)$ are called TRO equivalent 
if there exist TROs $ \cl M_1\subset B(K_1, H_1), \;\;\cl M_2\subset B(K_2,H_2)$ such that 
$$ \cl U_1=[\cl M_2\cl U_2\cl M_1^*], \;\;\;\;\cl U_2=[\cl M_2^*\cl U_2\cl M_1]. $$ We can see that TRO equivalent masa bimodules are 
spatially Morita equivalent.}
\end{examples}

\section{Bilattices and their homomorphisms}

\begin{theorem}\label{21} Let $\cl S$ be a commutative bilattice with the property that if $\{(P_i, Q_i)\}_i$ is an 
arbitrary subset of $\cl S$, then $$(\vee_i P_i, \wedge_i Q_i ), \;\;(\wedge_i P_i,  \vee_i Q_i)\in \cl S.$$ 
Then $\cl S$ is strongly closed.
\end{theorem}
\begin{proof} The proof is based on the fact that a commutative lattice  is strongly closed 
iff it is meet and join complete, \cite{ed}. By Theorem 3.1 in  \cite{st}

\begin{equation} \label{bil} Bil( \mathfrak {M}(\cl S)) \cap ( \cl S_l\times \cl S_r )=\cl S.\end{equation}

We can easily check that $\cl S_l, \cl S_r$ are complete lattices. Therefore they are CSLs. If $\{(P_i, Q_i)\}_i\subset \cl S$ is a net 
such that $P$  is the limit of the net $(P_i)_i$ and $Q$  is the limit of the net $(Q_i)_i,$ since 
$\cl S_l, \cl S_r$ are CSLs we have $(P, Q)\in \cl S_l\times \cl S_r. $ 
Now since $Q_i \mathfrak {M}(\cl S) P_i=0$ for all $i$, we have 
$Q \mathfrak {M}(\cl S) P=0.$  By (\ref{bil}), $(P, Q)\in \cl S.$ So $\cl S$ is strongly closed.
\end{proof}

\begin{definition}\label{22}
 Let $ \cl S^1, \cl S^2 $ be strongly closed commutative bilattices. 
If $$\phi:  \cl S_l^1\rightarrow 
 \cl S_l^2, \;\; \psi : \cl S_r^1\rightarrow  \cl S_r^2 $$ 
 are CSL homomorphisms such that
$$(\phi (P),\psi(Q) )\in \cl S^2\;\;\forall \;\;(P,Q)\in \cl S^1,$$ we say that 
$\phi \oplus \psi $ is a bilattice homomorphism. In the case where the $\phi ,\psi $ 
are 1-1 and onto, we call  $\phi \oplus \psi $ a bilattice isomorphism.

\end{definition}

Let $\cl D_i\subset B(H_i), i=1,2$ be masas. A bilattice $\cl S\subset B(H_1)\times B(H_2)$ is called a $\cl D_1\times \cl D_2$ bilattice 
if $\cl S_r\subset \cl D_2,  \cl S_l\subset \cl D_1.$ If $\cl U$ is a reflexive $\cl D_2-\cl D_1$ bimodule, the bilattice 
$$\cl S_0=\{(P, Q)\in \cl D_1\times \cl D_2: Q\cl UP=0\}$$ is a strongly closed commutative bilattice satisfying $\cl U=\mathfrak {M}(\cl S_0
) ,$ \cite{st}. We see that $\cl S_0$ is the biggest $\cl D_1\times \cl D_2$ bilattice $\cl S$ such that $\mathfrak {M}(\cl S)=
\cl U.$ We are going to find the smallest. 

Suppose that $\cl U$ is a nondegenerate reflexive $\cl D_2-\cl D_1$ bimodule and $\phi =\Map{\cl U}.$ Also suppose that $\cl L_i$ 
is the CSL generated by $\cl S_{i, \phi }, i=1,2.$ We define $$\cl S=\{(P, Q)\in \cl L_1\times \cl L_2^\bot : \phi (P)
\leq Q^\bot \}.$$

\begin{theorem}\label{23}

$(i)$ $\cl S$ is a strongly closed commutative bilattice satisfying $\cl S_l=\cl L_1, \cl S_r=\cl L_2^\bot , \cl U= \mathfrak {M}(\cl S) .$

$(ii)$ If $\cl S^1$  is a strongly closed commutative bilattice such that $\cl U=\mathfrak {M}(\cl S^1)$, then $\cl S\subset \cl S^1.$

\end{theorem}
 \begin{proof} 

$(i)$ We shall prove that $\cl S$ is a strongly closed commutative bilattice. By Theorem \ref{21}, it suffices to prove that if 
$(P_i, Q_i)_i$ is an 
arbitrary subset of $\cl S$, then $$ (\vee_i P_i, \wedge_i Q_i ), \;\;(\wedge_i P_i,  \vee_i Q_i)\in \cl S.$$ Indeed, since 
$ \phi (P_i)\leq Q^\bot _i $ for all $i$ and $\phi $ is sup preserving, we have 
$$\phi (\vee _iP_i)\leq (\wedge_i Q_i)^\bot\Rightarrow  (\vee_i P_i, \wedge_i Q_i )\in \cl S. $$ On the other 
hand, since $\phi $ is order preserving, we have 
$$\phi (\wedge  _iP_i)\leq \wedge _i\phi (P_i)\leq (\vee _i Q_i)^\bot\Rightarrow  (\wedge _i P_i, \vee _i Q_i )\in \cl S. $$ 
We shall prove that $\cl U= \mathfrak {M}(\cl S) .$ We have 
$$\mathfrak {M}(\cl S) =\{ T\in B(H_1, H_2):  QTP=0\;\;\forall \;\;(P,Q)\;\in \;\cl S\}\subset $$$$
\{T\in B(H_1, H_2):  \phi (P)^\bot TP=0\;\;\forall \;\;P\;\in \;\cl S_{1,\phi }\}=\cl U.$$
Also if $(P,Q)\in \cl S$, then $\phi (P)\leq Q^\bot $ so since $\phi (P)^\bot 
\cl UP=0\Rightarrow Q\cl UP=0.$ Thus,   $\cl U\subset \mathfrak {M}(\cl S) 
.$

$(ii)$ Let $\cl S^1$ be a strongly closed commutative bilattice such that $\cl U= \mathfrak {M}(\cl S^1) .$ 
If $(P,Q)\in \cl S^1$ then $Q\in \cl S_r^1.$ Therefore
$$ \mathrm{Alg}(\cl S_r^1) Q =Q \mathrm{Alg}(\cl S_r^1) Q   \Rightarrow Q \mathrm{Alg}(\cl S_r^1)^* =Q \mathrm{Alg}(\cl S_r^1) ^*Q   .$$
Thus, $$ Q\mathrm{Alg}(\cl S_r^1)^*\cl U P=Q \mathrm{Alg}(\cl S_r^1)^*Q\cl U P=0.$$
The conclusion is  
$\mathrm{Alg}(\cl S_r^1)^*\cl U\subset \cl U.$ So for an arbitrary projection $P$ we have 
$$\mathrm{Alg}( \cl S_r^1 )^*\cl UP\subset \cl UP\Rightarrow \phi (P)\in (\cl S_r^1)^\bot  .$$ Since $ \cl S_{2,\phi } =
\{\phi (P): P\in \cl P(B(H_1))\}$, we have 
$$ \cl S_{2,\phi } \subset ( \cl S_r^1 )^\bot  \Rightarrow  (\cl S_{2,\phi } )^\bot \subset \cl S_r^1 \Rightarrow 
\cl L_2^\bot \subset  \cl S_r^1 . $$ Similarly we can prove that $\cl L_1 \subset  \cl S_l^1 .$ 
Now if $$(P,Q)\in \cl S\Rightarrow Q\cl UP=0\Rightarrow (P,Q)\in Bil( \mathfrak {M}(\cl S^1) )\cap (\cl S_l^1 \times \cl S_r^1 ).$$ 
So by Theorem 3.1 in \cite{st} $(P,Q)\in \cl S^1.$
\end{proof}

\begin{definition}\label{24} We call $\cl S$ the essential bilattice of $\cl U.$
\end{definition}

\section{Morita equivalence of masa bimodules}

\begin{theorem}\label{31} Let $\cl S^1, \cl S^2$ be isomorphic, strongly closed commutative bilattices and $\cl U_1=\mathfrak {M}(\cl S^1) ,
\cl U_2=\mathfrak {M}(\cl S^2) .$ Then $\cl U_1$ and $\cl U_2$ are spatially Morita equivalent.
\end{theorem}
\begin{proof} Suppose that $\cl U_1\subset B(H_1, H_2)$ and $\cl U_2\subset B(K_1, K_2)$ are nondegenerate spaces and $\phi \oplus \psi :
 \cl S^1\rightarrow \cl S^2$ is a bilattice isomorphism. We define the CSLs 
$$\cl L_i=\{ P\oplus Q^\bot : (P, Q)\in \cl S^i\}, \;\;i=1,2.$$
The map $$\theta : \cl L_1\rightarrow \cl L_2: P\oplus Q^\bot \rightarrow\phi ( P)\oplus\psi ( Q)^\bot $$ is 
a CSL isomorphism. Therefore by Theorem \ref{12}, 
\begin{equation}\label{ex} \mathrm{Alg}(\cl L_1) =[Op(\theta ^{-1}) )Op(\theta )],\;\;\;\mathrm{Alg}(\cl L_2) =
[Op(\theta )Op(\theta^{-1} ) ].
\end{equation}
We also define the spaces:
$$\cl V_1=Op(\phi )\subset B(H_1, K_1)$$
$$\cl V_2=\{T\in B(H_2, K_2): \psi (Q)TQ^\bot =0\;\;\forall \;\;Q\;\in \;\cl S_r^1\}$$
$$\cl W_1=\{T\in B(K_1, H_1): P^\bot T\phi (P) =0\;\;\forall \;\;P\;\in \;\cl S_l^1\}$$
$$\cl W_2=\{T\in B(K_2, H_2): QT\psi (Q)^\bot =0\;\;\forall \;\;Q\;\in\; \cl S_r^1\}.$$
By Theorem \ref{12} we have 
$$ [\cl W_1\cl V_1]=\mathrm{Alg} (\cl S^1_l),  \;\;\;\;[\cl V_1\cl W_1]=\mathrm{Alg} (\cl S^2_l),   $$
$$ [\cl W_2\cl V_2]=\mathrm{Alg} (\cl S_r^l)^*,  \;\;\;\;[\cl V_2\cl W_2]=\mathrm{Alg} (\cl S^2_r)^*.   $$
So the above spaces are unital algebras. We also need the following strongly closed commutative bilattices:
$$\cl Z^1=\{(P, \psi (Q)): \;\;(P, Q)\;\;\in \;\;\cl S^1\}$$
$$\cl Z^2=\{(\phi (P), Q): \;\;(P, Q)\;\;\in \;\;\cl S^1\}.$$

A calculation shows that 
$$ Op(\theta ) =\left ( \begin{array}{clr} \cl V_1 & 0 \\  \mathfrak {M}(\cl Z^1) & \cl V_2
\end{array}\right),$$
$$ Op(\theta^{-1} ) =\left ( \begin{array}{clr} \cl W_1 & 0 \\  \mathfrak {M}(\cl Z^2) & \cl W_2
\end{array}\right).$$
Using (\ref{ex}) we obtain
\begin{equation}\label{exx} \cl U_1=[\mathfrak {M}(\cl Z^2) \cl V_1+\cl W_2\mathfrak {M}(\cl Z^1) ].
\end{equation}
We can see that $$\cl W_2\cl U_2\subset \mathfrak {M}(\cl Z^2) $$ and 
$$  \cl V_2\mathfrak {M}(\cl Z^2) \subset \cl U_2\Rightarrow \cl W_2\cl V_2 \mathfrak {M}(\cl Z^2) \subset \cl W_2\cl U_2. $$ 
Since $[\cl W_2\cl V_2] $ is a unital algebra, $\mathfrak {M}(\cl Z^2)=[ \cl W_2\cl U_2]. $ Similarly we can prove 
$\mathfrak {M}(\cl Z^1)=[ \cl U_2\cl V_1]. $  Thus by (\ref{exx}) $\cl U_1=[\cl W_2\cl U_2\cl V_1].$ Similarly we can prove 
$\cl U_2=[\cl V_2\cl U_1 \cl W_1].$ The proof is complete.\end{proof}

We are going to prove that two spatially Morita equivalent reflexive masa bimodules have isomorphic essential bilattices. In the rest of this section,
 we fix reflexive masa bimodules $\cl U_1\subset B(H_1, H_2), \;\;\;\cl U_2\subset B(K_1, K_2)$ and spaces 
$\cl V_1, \cl V_2, \cl W_1, \cl W_2$ satisfying conditions (i) and (ii) of Definition \ref{13}. Suppose that the essential bilattices of 
$\cl U_1$ and $\cl U_2$ are $\cl S^1$ and $\cl S^2$, respectively. We write
$$\phi_i=\Map{\cl U_i},\;\;\; \psi_i=\Map {\cl W_i} , \;\;\; \chi_i=\Map{\cl V_i} , i=1,2.$$

\begin{lemma}\label{32} 

$(i)$ $$\chi _i( \Lat {[\cl W_i\cl V_i]} ) =\Lat{[\cl V_i \cl W_i]}, \;\;i=1,2.$$

$(ii)$ $$\psi _i(\Lat {[\cl V_i\cl W_i]})=\Lat{[\cl W_i \cl V_i]}, \;\;i=1,2.$$

\end{lemma}

\begin{proof} Choose $P\in \Lat {[ \cl W_1\cl V_1 ]} .$ We have 
$$P^\bot \cl W_1\cl V_1 P=0\Rightarrow P^\bot \cl W_1\chi _1(P)=0.$$
Now $$ \chi _1(P)^\bot \cl V_1\cl W_1 \chi _1(P) = \chi _1(P)^\bot \cl V_1P^\bot \cl W_1 \chi _1(P) =0.$$ Thus 
$\chi _1(P)\in \Lat{ [\cl V_1\cl W_1] }.$
We have proved that $$\chi _1( \Lat {[\cl W_1\cl V_1]} )\subset \Lat{[\cl V_1 \cl W_1]}.$$ Similarly we can prove 
$$\psi  _1( \Lat {[\cl V_1\cl W_1]} )\subset \Lat{[\cl W_1 \cl V_1]}$$ which implies 
 $$\chi _1(\psi  _1( \Lat {[\cl V_1\cl W_1]} ))\subset \chi _1(\Lat{[\cl W_1 \cl V_1]})$$
Suppose that $\zeta =\mathrm{Map}(   [\cl V_1\cl W_1]   ).$ We can easily check that $\zeta =\chi_1\circ  \psi_1 $ 
and since $[\cl V_1\cl W_1]   $ is unital algebra $\zeta (Q)=Q, \;\forall \;Q\; \in \; \Lat {[\cl V_1\cl W_1]} . $
Thus $\chi_1( \psi_1(Q))=Q, \;\forall \;Q\; \in \; \Lat {[\cl V_1\cl W_1]} $ 
which implies $$ \Lat {[\cl W_1\cl V_1]} \subset \chi _1(\Lat {[\cl W_1\cl V_1]} ).$$ 
The conclusion is that  $$ \Lat {[\cl W_1\cl V_1]} =\chi _1(\Lat {[\cl W_1\cl V_1]} ).$$ 
The other proofs are similar. 
\end{proof}

\begin{lemma}\label{33} 

 Let $(P_i)_i\subset \Lat{[\cl W_1\cl V_1]}.$ Then $\wedge_i\chi _1(P_i) =\chi _1(\wedge_iP_i). $
\end{lemma}
\begin{proof}  Let $(P_i)_i\subset \Lat{[\cl W_1\cl V_1]}.$ By Lemma \ref{32} we have $(\chi _1(P_i))_i\subset \Lat{[\cl V_1\cl W_1]}.$ 
So $ \wedge _i\chi _1(P_i) \in \Lat{[\cl V_1\cl W_1]}.$ Thus there exists $P\in \Lat{[\cl W_1\cl V_1]}$ such that 
$ \wedge _i\chi _1(P_i) =\chi _1(P).$ Now we have
$$P=\psi_1 (\chi_1(P))=\psi _1( \wedge _i\chi _1(P_i) ) \leq \wedge _i \psi_1( \chi_1(  P_i)) =\wedge _iP_i.$$
It follows that $$\chi _1(P)\leq \chi _1(\wedge _iP_i)\Rightarrow \wedge _i\chi _1(P_i) \leq  \chi _1(\wedge _iP_i) \Rightarrow 
\wedge _i\chi _1(P_i) =  \chi _1(\wedge _iP_i). $$ 
\end{proof}

\begin{lemma}\label{34} If $(P_i)_i\subset \cl S^1_l$  and $(Q_i)_i\subset \cl S^1_r$, then 
  $$\wedge_i\chi _1(P_i) =\chi _1(\wedge_iP_i), \;\;\; \wedge_i\psi^*  _2(Q_i) =\psi  _2^*(\wedge_iQ_i) .$$

\end{lemma}
\begin{proof} Let  $(P_i)_i\subset \cl S^1_l$.  We shall prove that $\wedge_i\chi _1(P_i) =\chi _1(\wedge_iP_i). $ 
 By Lemma \ref{33}, it suffices to prove that $$\cl S_{1,\phi _1}\subset \Lat{[\cl W_1\cl V_1]}.$$ 
Since $$\cl U_1=\cl U_1[ \cl W_1 \cl V_1 ]$$ we have 
$$\cl U_1^*= [\cl W_1 \cl V_1 ]^* \cl U_1^*. $$ 
It follows from the definition of $\cl S_{2, \phi _1^*}$ that 
$\cl S_{2, \phi _1^*}\subset \Lat{ [\cl W_1 \cl V_1 ]^*}$ and hence $$\cl S_{1, \phi _1} \subset \Lat{[\cl W_1\cl V_1]}.$$
The other proof is similar.\end{proof}

\begin{lemma}\label{35} 

$(i)$  $$\cl S^2=\{(\chi_1(P) , \psi_2^*(Q )): \;\;(P,Q)\;\in \;\cl S^1\}$$

$(ii)$ $$\cl S^1=\{(\psi _1(P) , \chi _2^*(Q )): \;\;(P,Q)\;\in \;\cl S^2\}$$

\end{lemma}
\begin{proof} By Theorem \ref{21},  Lemma \ref{34}, and the fact that $\chi _1, \psi _2^*$ are $\vee$-continuous, 
the set $$\cl L_2=\{(\chi_1(P) , \psi_2^*(Q ): \;\;(P,Q)\;\in \;\cl S^1\}$$ is 
a strongly closed commutative bilattice. We need to show that $\cl S^2=\cl L_2.$ If $(P,Q)\in \cl S^1$, then 
as $\cl W_2\cl U_2 \cl V_1 \subset \cl U_1$,  
$$Q\cl U_1P=0\Rightarrow Q\cl W_2\cl U_ 2 \cl V_1P=0\Rightarrow \psi_2^*(Q) \cl U_2\chi_1(P) =0.$$ So 
$\cl U_2\subset \mathfrak M (\cl L_2 ) .$ When $X\in \mathfrak M(\cl L_2 ) $ and $(P,Q)\in \cl S^1$, we have 
$$\psi_2^*(Q) X\chi_1(P)=0\Rightarrow Q \cl W_2X\cl V_1 P=0\Rightarrow  \cl W_2X\cl V_1 \subset \cl U_1.$$
So  $$ [\cl V_2\cl W_2] X[\cl V_1\cl W_1]\subset \cl V_2\cl U_1\cl W_1\subset \cl U_2.$$ 
Since $[\cl V_2\cl W_2] $ and $[\cl V_1\cl W_1] $  are unital algebras, $X\in \cl U_2.$ We have proved that 
$\cl U_2=\mathfrak M(\cl L_2 ) .$ Theorem \ref{23} implies that $\cl S^2\subset \cl L_2.$ We have that 
\begin{equation}\label{ex1} \chi_1 \oplus \psi_2^*(\cl S^1)=\cl L_2\supset \cl S^2 
\end{equation}
and similarly 
\begin{equation}\label{ex2}\psi _1 \oplus \chi _2^*(\cl S^2)=\cl L_1\supset \cl S^1. \end{equation}
If $P\in S_l^1$, then $P\in \Lat{[\cl W_1\cl V_1]}$ and so $\psi_1( \chi_1(P))=P, $ (see the proof of Lemma \ref{32}).
Similarly if $Q\in \cl S_r^1$, then $\chi_2^*( \psi_2^*(Q)) =Q.$ So 
\begin{equation}\label{ex3} \cl L_1=(\psi_1 \oplus \chi_2^* )(\cl S^2)\subset (\psi_1 \oplus \chi_2^* )
(\chi_1 \oplus \psi_2^* )(\cl S^1)=\cl S^1.
\end{equation}
Similarly
\begin{equation}\label{ex4}
\cl L_2= (\chi _1 \oplus \psi_2^* ) (\cl S^1)\subset(\chi _1 \oplus \psi_2^* )  (\psi_1 \oplus \chi_2^* )(\cl S^2)=\cl S^2.\end{equation}

Equations (\ref{ex1}), (\ref{ex2}), (\ref{ex3}), and (\ref{ex4}) imply that $\cl L_1=\cl S^1, \cl L_2=\cl S^2.$

\end{proof}

\begin{theorem}\label{36} Two reflexive masa bimodules are spatially Morita equivalent if and only if their 
essential bilattices are isomorphic.
\end{theorem}
\begin{proof} The one direction is Theorem \ref{31}. For the other we use Lemma \ref{35} and the fact that
$$\chi_1 (\psi_1(P))=P \;\;\forall P\;\in \;\cl S_l^2 , \;\;\;\psi_1 (\chi_1(P))=P \;\;\forall \;P\;\in \cl S_l^1, $$
$$\chi_2^* (\psi_2^*(Q))=Q \;\;\forall Q\;\in \;\cl S_r^1 , \;\;\;\psi_2^*( \chi_2^*(Q))=Q \;\;\forall \;Q\;\in \cl S_r^2. $$
\end{proof}

\begin{corollary}Spatial Morita equivalence of reflexive masa bimodules is an equivalence relation.
\end{corollary}

\begin{proposition}\label{37} Two CSL algebras are spatially Morita equivalent as algebras iff they are spatially Morita equivalent 
as spaces.
\end{proposition}\begin{proof}
Suppose that $\cl A_1=\mathrm{Alg}(\cl L_1)$ and $\cl A_2=\mathrm{Alg}(\cl L_2)$  are CSL algebras. Their 
essential bilattices are
$$\cl S^i=\{(P,Q)\in \cl L_i\times \cl L_i^\bot : \;\;P\leq I-Q\}.$$
We can easily  see that $\cl L_1, \cl L_2$ are isomorphic CSLs iff the bilattices $\cl S^1, \cl S^2$ are isomorphic.\end{proof}

\section{Continuous bilattice homomorphisms}

\begin{definition}\label{411} Let $\cl S^1, \cl S^2$ be strongly closed commutative bilattices and 
$ \cl U_1=\mathfrak M (\cl S^1) , \cl U_2=\mathfrak M (\cl S^2) .$ If $$\phi \oplus \psi : 
\cl S^1\rightarrow \cl S^2$$ is a bilattice homomorphism, we define the spaces
$$\cl V_1=\{T: \phi (P)^\bot TP=0\;\;\forall \;\;P\in \cl S_l^1\}$$
$$\cl V_2=\{T: \psi  (Q) TQ^\bot =0\;\;\forall \;\;Q\in \cl S_r^1\}$$
$$\cl W_1=\{T: P^\bot T\phi (P)=0\;\;\forall \;\;P\in \cl S_l^1\}$$
$$\cl W_2=\{T: Q T\psi (Q)^\bot =0\;\;\forall \;\;Q\in \cl S_r^1\}.$$
\end{definition}

We isolate from the proof of Theorem \ref{31} the following lemma.

\begin{lemma}\label{41} Let $\phi, \psi, \cl U_1, \cl U_2, \cl V_1, \cl V_2, \cl W_1, \cl W_2 $ be as in the above definition.  
 If $\phi \oplus \psi $ is an isomorphism, then
$$ \cl U_1=[\cl W_2\cl U_2 \cl V_1] , \;\;\;\cl U_2=[\cl V_2\cl U_1 \cl W_1] $$
$$ [\cl W_1\cl V_1]=\Alg {\cl S_l^1} ,\;\;\;[\cl V_1\cl W_1]=\Alg {\cl S_l^2} $$
$$ [\cl W_2\cl V_2]=\Alg {\cl S_r^1}^* ,\;\;\;[\cl V_2\cl W_2]=\Alg {\cl S_r^2}^*. $$
\end{lemma}

\begin{theorem}\label{42} Let $\cl S^1, \cl S^2$ be strongly closed commutative bilattices and suppose that $$\phi \oplus \psi : 
\cl S^1\rightarrow \cl S^2$$ is a continuous onto bilattice homomorphism. We recall 
 the spaces $$\cl U_1, \cl U_2, 
\cl V_1, \cl V_2, \cl W_1, \cl W_2 $$ from Definition \ref{411}. Then 
$$ \cl U_1\supset \cl W_2\cl U_2 \cl V_1 , \;\;\; \cl U_2=[\cl V_2\cl U_1 \cl W_1] $$
$$ \cl W_1\cl V_1\subset \Alg {\cl S_l^1} ,\;\;\;[\cl V_1\cl W_1]=\Alg {\cl S_l^2} $$
$$ \cl W_2\cl V_2\subset \Alg {\cl S_r^1}^* ,\;\;\;[\cl V_2\cl W_2]=\Alg {\cl S_r^2}^*. $$

\end{theorem}
\proof We define the bilattices 
$$ \cl Z^1=\{ (P\oplus P, Q\oplus Q) : \;\;(P,Q)\in \cl S^1\}$$
$$\cl Z^2=\{ (P\oplus \phi (P), Q\oplus \psi (Q)) : \;\;(P,Q)\in \cl S^1\}.$$
We also define the masa bimodules 
$$\hat{\cl U_1}=\mathfrak M(\cl Z^1)= \left (\begin{array}{clr} \cl U_1 & \cl U_1 \\ \cl U_1 & \cl U_1 \end{array}\right) ,$$ 

$$\hat{\cl U_2}=\mathfrak M(\cl Z^2)=\left (\begin{array}{clr} \cl U_1 & \Omega _1 \\ \Omega _2 & \cl U_2
 \end{array}\right),$$ where 
$$\Omega _1=\{T: QT\phi (P)=0\;\;\forall \;(P,Q)\in \cl S^1\}, $$ 
$$\Omega _2=\{T: \psi (Q)T P=0\;\;\forall \;(P,Q)\in \cl S^1\}. $$
We define the bilattice isomorphism
$$ \hat{\phi }\oplus \hat{\psi }: \cl Z^1\rightarrow \cl Z^2: (P\oplus P, Q\oplus Q)\rightarrow  
(P\oplus \phi (P), Q\oplus \psi (Q)) .$$ 
We also define the spaces
$$\hat{\cl V_1}=\{T: \hat{\phi }(L)^\bot TL=0\;\;\forall L\in \cl Z_l^1\}=
\left (\begin{array}{clr} \Alg{S_l^1} & \Alg{S_l^1}  \\ \cl V_1 & \cl V_1 \end{array}\right) ,$$
$$\hat{\cl V_2}=\{T: \hat{\psi  }(L) TL^\bot =0\;\;\forall L\in \cl Z_r^1\}=
\left (\begin{array}{clr} \Alg{S_r^1}^* & \Alg{S_r^1}^*  \\ \cl V_2 & \cl V_2 \end{array}\right) ,$$
$$\hat{\cl W_1}=\{T: L^\bot T\hat{\phi }(L)=0\;\;\forall L\in \cl Z_l^1\}=
\left (\begin{array}{clr} \Alg{S_l^1} & \cl W_1 \\ \Alg{S_l^1}   & \cl W_1 \end{array}\right) ,$$
$$\hat{\cl W_2}=\{T: L T\hat{\psi  }(L)^\bot =0\;\;\forall L\in \cl Z_r^1\}=
\left (\begin{array}{clr} \Alg{S_r^1}^* & \cl W_2 \\ \Alg{S_r^1}^*   & \cl W_2 \end{array}\right) .$$
Since $\hat{\phi }\oplus \hat{\psi }$ is a bilattice isomorphism, from Lemma \ref{41} 
we have
 $$\hat{\cl U_2}=[\hat{\cl V_2}\hat{\cl U_1}\hat{ \cl W_1}]\Rightarrow  \cl U_2=[\cl V_2\cl U_1 \cl W_1] ,$$ 
 $$\hat{\cl U_1}=[\hat{\cl W_2}\hat{\cl U_2}\hat{ \cl V_1}]\Rightarrow  \cl U_1\supset \cl W_2\cl U_2 \cl V_1.$$
We also define the algebras
$$ \Alg{\cl Z_l^1} =\left (\begin{array}{clr} \Alg{S_l^1} & \Alg{S_l^1}  \\ \Alg{S_l^1} & \Alg{S_l^1}   \end{array}\right), $$ 
$$\Alg{\cl Z_l^2}=\left (\begin{array}{clr} \Alg{S_l^1} & \cl W_1   \\ \cl V_1  & \Alg{S_l^2}   \end{array}\right), $$ 
$$\Alg{\cl Z_r^1}=\left (\begin{array}{clr} \Alg{S_r^1} & \Alg{S_r^1}  \\ \Alg{S_r^1} & \Alg{S_r^1}   \end{array}\right), $$ 
$$\Alg{\cl Z_r^2}=\left (\begin{array}{clr} \Alg{S_r^1} & \cl V_2^*   \\ \cl W_2^*  & \Alg{S_r^2}   \end{array}\right). $$ 
By Lemma \ref{41}, we have 
$$\Alg{\cl Z_l^1} =[\hat{\cl W_1}\hat{\cl V_1}]\Rightarrow \cl W_1\cl V_1\subset \Alg{\cl S_l^1}$$
$$\Alg{\cl Z_l^2} =[\hat{\cl V_1}\hat{\cl W_1}]\Rightarrow [\cl V_1\cl W_1]= \Alg{\cl S_l^2}$$
$$\Alg{\cl Z_r^1}^* =[\hat{\cl W_2}\hat{\cl V_2}]\Rightarrow \cl W_2\cl V_2\subset \Alg{\cl S_r^1}^*$$
$$\Alg{\cl Z_r^2}^* =[\hat{\cl V_2}\hat{\cl W_2}]\Rightarrow [\cl V_2\cl W_2]= \Alg{\cl S_r^2}^*.\qquad \Box$$

\begin{lemma}\label{43} Let $\cl L_1, \cl L_2$ be CSLs acting on separable Hilbert spaces, $\theta : \cl L_1\rightarrow \cl L_2$ be a CSL 
isomorphism, and $\cl U=Op(\theta ), \cl V=Op(\theta ^{-1})$. Then 
$$ \Alg{\cl L_1}_{min} =[ \cl V_{min}\cl U_{min} ] , \;\;\;\Alg{\cl L_2}_{min}=[\cl U_{min}\cl V_{min}]. $$  
\end{lemma}
\proof By Theorem \ref{11} we have $$ \Alg{\cl L_1} =[\cl V\cl U] , \;\;\;\Alg{\cl L_2}=[\cl U\cl V]. $$
By \cite{stwt} the spaces $\cl V_{min}, \cl U_{min}$ are equal to the weak* closure of the pseudointegral operators which contain. 
Since the product of two pseudointegral operators is also a pseudointegral operator \cite{dav},
$$ \cl V_{min}\cl U_{min} \subset \Alg{\cl L_1}_{min} .$$ 
It remains to show $$ \Alg{\cl L_1}_{min} \subset[\cl V_{min}\cl U_{min} ].$$ 
Since $\Alg{\cl L_1}_{min} $ is the smallest masa bimodule $\cl X$ containing in $\Alg{\cl L_1}$ 
such that $ \mathrm{Ref}(\cl X) =\Alg{\cl L_1},$ it suffices to prove that 
$\Alg{\cl L_1}=  \mathrm{Ref}(\cl V_{min}\cl U_{min} ) $ or equivalently $$\cl V\cl U\subset \mathrm{Ref}(\cl V_{min}\cl U_{min} )  .$$ 
We fix operators $A\in \cl V, B\in \cl U$ and an arbitrary vector $\xi .$ Since 
$\mathrm{Ref}(\cl U_{min} )=\cl U ,$ we have $$B(\xi )\in \overline{\cl U_{min}\xi }$$ 
and since $\mathrm{Ref}(\cl V_{min} )=\cl V ,$ we have $$A(B(\xi ))\in \overline{\cl V_{min}B(\xi) }.$$
Thus  $$A(B(\xi ))\in \overline{\cl U_{min}\cl V_{min}\xi }.$$ We proved that $AB\in  \mathrm{Ref}(\cl V_{min}\cl U_{min} ) .$
The conclusion is that  $ \Alg{\cl L_1 } =[ \cl V_{min}\cl U_{min} ]  $ and similarly 
$\Alg{\cl L_2}_{min} =[ \cl U_{min}\cl V_{min} ] . \qquad \Box$

\begin{lemma}\label{44} Let  $\cl L_1, \cl L_2$ be CSLs acting on separable Hilbert spaces and $\theta : \cl L_1\rightarrow \cl L_2$ be an onto continuous  CSL 
homomorphism. Put
$$\cl U=\{T: \theta (P)^\bot TP=0\;\;\forall P\in \cl L_1\}$$
$$\cl V=\{T: P^\bot T\theta (P)=0\;\;\forall P\in \cl L_1\}.$$ 
Then $\Alg{\cl L_2}_{min} =[ \cl U_{min}\cl V_{min} ].$\end{lemma}
\proof We define the CSLs 
$$ \cl N_1=\{P\oplus P: P\in \cl L_1\} , \;\;\;\;\cl N_2=\{P\oplus \theta (P): P\in \cl L_1\}. $$
The map $$\rho : \cl N_1\rightarrow \cl N_2: P\oplus P\rightarrow P\oplus \theta (P)$$ is 
a CSL isomorphism. We can see that 
$$ \Alg{\cl N_1} =\left (\begin{array}{clr} \Alg{\cl L_1} & \Alg{\cl L_1}  \\ \Alg{\cl L_1} & \Alg{\cl L_1}   \end{array}\right), $$ 
$$ \Alg{\cl N_2} =\left (\begin{array}{clr} \Alg{\cl L_1} & \cl V  \\ \cl U & \Alg{\cl L_2}   \end{array}\right), $$ 
$$ Op(\rho ) =\left (\begin{array}{clr} \Alg{\cl L_1} & \Alg{\cl L_1}  \\ \cl U & \cl U   \end{array}\right), $$ 
$$ Op(\rho^{-1} ) =\left (\begin{array}{clr} \Alg{\cl L_1} & \cl V  \\ \Alg{\cl L_1} & \cl V     \end{array}\right). $$
Since Lemma \ref{43} implies $$\Alg{\cl N_2} _{min}=[Op(\rho )_{min} Op(\rho^{-1} )_{min} ],$$ 
we have  $\Alg{\cl L_2}_{min} =[ \cl U_{min}\cl V_{min} ]. \qquad \Box$

\begin{theorem}\label{45} Let $\cl S^1, \cl S^2$ be strongly closed commutative bilattices acting on separable Hilbert spaces and $$\phi \oplus \psi : 
\cl S^1\rightarrow \cl S^2$$ be a continuous onto bilattice homomorphism. If $ \cl U_1=\mathfrak M (\cl S^1) , 
\cl U_2=\mathfrak M (\cl S^2) $ and $\cl U_1$ is synthetic, then $\cl U_2$ is synthetic. 
\end{theorem}
  \proof Let $\cl V_1, \cl V_2, \cl W_1, \cl W_2$ be as in Theorem \ref{42}. Since 
$ \cl W_2\cl U_2\cl V_1\subset  \cl U_1 $, we have 
$$ (\cl W_2)_{min}\cl U_2(\cl V_1)_{min} \subset \cl U_1  \Rightarrow$$$$ 
(\cl V_2)_{min} (\cl W_2)_{min} \cl U_2(\cl V_1)_{min} (\cl W_1)_{min } \subset (\cl V_2)_{min} \cl U_1  (\cl W_1)_{min } 
.$$ Since $\cl U_1= (\cl U_1)_{min}$ and $\cl V_2\cl U_1\cl W_1\subset \cl U_2$ we have 
$ (\cl V_2)_{min}\cl U_1(\cl W_1)_{min}\subset (\cl U_2)_{min}.$ Thus 
$$(\cl V_2)_{min} (\cl W_2)_{min} \cl U_2(\cl V_1)_{min} (\cl W_1)_{min } \subset (\cl U_2)_{min}.$$ 
By Lemma \ref{44} $$ \Alg{S_r^2}^*_{min}=[(\cl V_2)_{min} (\cl W_2)_{min}] \;\;\;  
\Alg{S_l^2}_{min}  =[(\cl V_1)_{min} (\cl W_1)_{min}] .$$ Therefore 
$\cl U_2\subset \Alg{S_r^2}^*_{min}\cl U_2\Alg{S_l^2}_{min}  \subset (\cl U_2)_{min}.$ The proof is complete. $\qquad \Box$  

\begin{theorem}\label{46} Let $\cl S^1, \cl S^2$ be strongly closed commutative bilattices and suppose that $$\phi \oplus \psi : 
\cl S^1\rightarrow \cl S^2$$ is a continuous onto bilattice homomorphism. We denote $ \cl U_1=\mathfrak M (\cl S^1) , 
\cl U_2=\mathfrak M (\cl S^2) .$

A. If $\cl U_2$  
contains a nonzero compact (or nonzero finite rank or rank one) 
operator, then $\cl U_1$  
contains a nonzero compact (or nonzero finite rank or rank one) 
operator. 

B. If $\cl S^1$ and $\cl S^2$ act on separable Hilbert spaces and $(\cl U_2)_{min}$ 
contains a nonzero compact (or nonzero finite rank or rank one) 
operator, then  $(\cl U_1)_{min}$ 
contains a nonzero compact (or nonzero finite rank or rank one) 
operator. 

\end{theorem}
\proof 

A. Let $\cl V_1, \cl V_2, \cl W_1, \cl W_2$ be as in Theorem \ref{42}. 
Suppose that $K\in \cl U_2$ is a nonzero compact operator. We shall prove that 
$\cl U_1$ also contains a nonzero compact operator. It 
suffices to prove that $[ \cl W_2K\cl V_1]\neq 0 .$ Indeed, if 
$[ \cl W_2K\cl V_1 ]= 0$, then $[\cl V_2\cl W_2] K [\cl V_1 \cl W_1] = 0.$
Since by Theorem \ref{42}, $[\cl V_2\cl W_2] $ and $[\cl V_1 \cl W_1] $ are unital, $K=0.$ This is a contradiction.

B. Suppose that $\cl S^1$ and $\cl S^2$ act on separable Hilbert spaces and $K\in (\cl U_2)_{min}$ is a nonzero compact operator. Since 
$\cl W_2\cl U_2\cl V_1\subset \cl U_1$ we have $(\cl W_2)_{min}K(\cl V_1)_{min}\subset (\cl U_1)_{min}.$
 It 
suffices to prove that $[ (\cl W_2)_{min}K(\cl V_1)_{min}]\neq 0 .$ Indeed, if 
$[ (\cl W_2)_{min}K(\cl V_1)_{min} ]= 0$, then $[(\cl V_2)_{min}(\cl W_2)_{min}] K [(\cl V_1)_{min} (\cl W_1)_{min}] = 0.$
By Lemma \ref{44}, $[(\cl V_2)_{min}(\cl W_2)_{min}] $ and $[(\cl V_1)_{min} (\cl W_1)_{min}] $ are unital, thus $K=0.$ This is also 
a contradiction. The proofs of the other
 assertions are similar. $\qquad \Box$

\bigskip
As was proved in \cite{ele2}, onto homomorphisms between CSLs have the properties 
described in Theorems \ref{45} and \ref{46}.

\section{Inverse results }

In this section we prove the inverse image theorem, \cite{st}, using the results of this paper 
(Theorem \ref{52}), and we present a result about the problem of when do the Borel sets 
satisfying the assumptions of the inverse image theorem support a nonzero compact (or  finite rank, or 
rank 1) operator (Theorem \ref{54}).  As we will show, Theorems \ref{52} and \ref{54} are special cases  
of the theory developed in this paper: the assumptions of the inverse image theorem 
can be used to define a homomorphism between the appropriate bilattices and then one can apply Theorems 
\ref{45} and \ref{46}.
  
We need some notions and facts. Let $(X, \mu ), (Y, \nu )$ be standard measure spaces, that is, the measures $\mu $ and $\nu $ 
are regular Borel measures with respect to some Borel structures on $X$ and $Y$ arising from complete metrizable topologies. Put
$H_1=L^2(X,\mu )$ and $H_2=L^2(Y, \nu ).$ For a function $\phi \in L^\infty (X, \mu ),$ let $M_\phi $ be the operator 
on $H_1$ given by $M_\phi (f)=\phi f, f\in H_1,$ and similarly define $M_\psi $ for $\psi \in L^\infty (Y,\nu ).$ Put 
$$ \cl D_1=\{M_\phi , \phi \in L^\infty (X, \mu )\}, 
\cl D_2=\{M_\psi  , \psi  \in L^\infty (Y, \nu  ) \}.$$ These 
algebras are masas.  We need several facts and notions from the theory of masa bimodules, \cite{arv, eks, st}.
A subset $E\subset X\times Y$ is called \textit{marginally null} if $E\subset (M_1\times Y)\cup (X\times M_2)$ 
where $\mu (M_1)=\nu(M_2)=0. $ We call two subsets $E, F\subset X\times Y$ \textit{marginally equivalent} (and write $ E\cong F$) 
if the symmetric difference of $E$ and $F$ is marginally null. A set $\sigma \subset X\times Y$ 
is called $\omega-$\textit{open} if it is 
marginally equivalent to a countable union of the form $\cup _{i=1}^\infty \alpha _i\times \beta _i$, where
$\alpha _i\subset X$ and $\beta _i\subset Y$ are measurable sets for $i\in \bb N.$ The complement of an $\omega $-open set 
is called $\omega $-\textit{closed}. An operator $T\in B(H_1, H_2)$ is said to be \textit{supported on} $\sigma $ if 
$M_{\chi _\beta }TM_{\chi _\alpha }=0$ whenever $(\alpha \times \beta)\cap  \sigma \cong \emptyset. $
(Here $\chi _\gamma $ stands for the characteristic function of the measurable subset $\gamma $). 
 Given a 
$\sigma \subset X\times Y$, let $$\frak M_{max}(\sigma )=\{T\in B(H_1, H_2): T \;\;\mbox{supported\;\;on\;\;}\sigma \}.$$
The space $\frak M_{max}(\sigma )$ is a reflexive masa bimodule, Theorem 4.1 in \cite{eks}.

For Lemma \ref{51}, we fix an $\omega $-closed set $k.$ 
 If $\alpha \subset X$ (resp. $\beta \subset Y$), we denote by $P(\alpha )$ (resp. $Q(\beta )$ ) the 
projection onto $L^2(\alpha, \mu )$ (resp. $L^2(\beta, \nu )$). We denote 
by $\cl U$ the space $\frak M_{max}(k).$ By the proof of Theorem 4.2 in \cite{eks} there exist Borel sets
 $\alpha_n\subset X,  \beta_n\subset Y $ such that $\cl U=\frak M_{max}(
X\times Y\setminus (\cup _n \alpha_n \times \beta_n ))$ and such that 
$$T\in \cl U\Leftrightarrow Q(\beta _n)TP(\alpha _n)=0,\;\;\forall \;n\;\in \;\bb N.$$
Since $k$ and $X\times Y\setminus (\cup _n \alpha_n \times \beta_n )$ are $\omega $-closed sets, they are marginally equivalent, \cite{eks}.
Suppose that $\cl S$ is the commutative strongly closed bilattice generated by 
the set $$\{(P(\alpha_n),  Q(\beta_n) ): n\in \bb N\}.$$

\begin{lemma}\label{51} $$ \cl U= \mathfrak M(\cl S) .$$
\end{lemma}  
\proof Suppose that $\phi =\Map {\cl U}.$ If $T\in \mathfrak M(\cl S) $, then $Q(\beta_n)T P(\alpha _n)=0$ for all $n.$ 
 So $T\in \cl U, $ and thus $\cl U\supset  \mathfrak M(\cl S) .$ For the converse, we define the set 
$$\cl L=\{ (P(\alpha ), Q(\beta )) : \;\;(P(\alpha ), Q(\beta ))\in \cl S , \;\;\phi(P( \alpha))\leq 
Q( \beta)^ \bot \}.$$
Since $\phi $ is sup preserving we can prove that $\cl L$ is a strongly closed bilattice. Since 
$$\{(P(\alpha _n), Q(\beta _n)): \;\;n \;\;\in \;\;\bb N \}\subset \cl L$$ we have $\cl L\supset \cl S.$  
So if $T\in \cl U$ and $(P(\alpha ), Q(\beta )) \in \cl S$, then $\phi (P(\alpha ))\leq Q(\beta )^\bot .$ 
Thus as $\phi (P(\alpha ))^\bot TP(\alpha )=0$ we obtain $Q(\beta )TP(\alpha )=0.$ We have proved that $\cl U= \mathfrak M(\cl S) . \qquad \Box$

\medskip

We call $\cl S$ the $k$ \textit{generated bilattice}.

\medskip

If $(X, \mu )$ and $(Y, \nu )$ are Borel spaces, an $\omega $-closed  subset $\sigma \subset X\times Y$  is called a set of 
$\mu \times \nu $ synthesis if the masa bimodule $$\frak M_{max}(\sigma  )\subset B(L^2(X,\mu ), L^2(Y,\nu ))$$ 
is synthetic.

\begin{theorem} \label{52} (Inverse Image Theorem) Suppose that $ (X, \mu ), (Y, \nu ), (X_1, \mu_1 ), (Y_1, \nu_1 ), $ 
 are Borel spaces and that $\theta : X\rightarrow X_1, \;\;\;\rho : Y\rightarrow Y_1$ are Borel mappings. 
We denote by $\theta _*(\mu ), \rho _*(\nu )$ 
the measures given by $\theta_*( \mu)(\alpha ) =\mu (\theta ^{-1}(\alpha )), $ $\rho_*( \nu )(\beta )=
\nu (\rho ^{-1}(\beta ))$ for every Borel set $\alpha \subset X_1, $ $
 \beta \subset Y_1$  and we assume that they are absolutely continuous with respect to measures $\mu _1$ and $\nu _1$ 
respectively. If $k_1$ is an $\omega -$closed set of $\mu _1\times \nu _1$ synthesis, then $k=(\theta \times \rho )^{-1}(k_1)$ 
is a  set of $\mu \times \nu $ synthesis.
\end{theorem}
\proof Put $ \cl U_1=\frak M_{max}(k_1) , \;\;\cl U=\frak M_{max}(k) .$ Suppose that 
 $$k_1=X_1\times Y_1\setminus (\cup _n \alpha_n \times \beta_n ).$$ Let 
$\cl S^1$ be the $k_1$ generated bilattice. 
Since $\theta _*(\mu ),$ (resp.  $\rho _*(\nu )$), is absolutely continuous with respect to $\mu _1,$  
(resp. $\nu _1$) the  maps 
$$\cl P( L^\infty (X_1, \mu _1)) \rightarrow \cl P(L^\infty (X, \mu))
: P_1(\alpha ) \rightarrow P(\theta ^{-1}(\alpha )), $$
 $$\cl P( L^\infty (Y_1, \nu  _1)) \rightarrow \cl P(L^\infty (Y, \nu )): P_1(\beta  ) 
\rightarrow P(\rho ^{-1}(\beta  
)) ,$$ are well defined and they can extend to  weak* continuous homomorphisms between the corresponding masas. Put $\cl S=\hat{\theta }\oplus \hat{\rho }(\cl S^1).$ By Theorem 
\ref{45}, it suffices to prove that $\cl U=\mathfrak M(\cl S) .$

Since, by Lemma \ref{51}, $\cl S^1$ is the bilattice generated by $$\{(P_1(\alpha_n), Q_1( \beta_n) ): \;\;n\;\;\in \bb N\},$$ 
$\cl S$ is the bilattice generated by $$\{( P(\theta ^{-1}(\alpha_n)) , Q( \rho ^{-1}(\beta_n) )  ): \;\;n\;\;\in \bb N\}. 
$$ So 
$$\mathfrak M(\cl S) =\{T: Q( \rho ^{-1}(\beta_n))  TP( \theta ^{-1}(\alpha_n) ) =0\;\;\forall \;\;n\in \bb N\}=$$$$
\frak M_{max}(X\times Y\setminus (\cup _n\theta ^{-1}(\alpha_n) \times \rho ^{-1}(\beta_n) ))=\cl U.\qquad \Box$$ 

\medskip

The proof of the following theorem is similar to that of Theorem \ref{52}, using Theorem \ref{46}.

\begin{theorem} \label{54}  Let  $ (X, \mu ), (Y, \nu ), (X_1, \mu_1 ), (Y_1, \nu_1 ), $ 
 be  Borel spaces such that $\theta_*( \mu)$ (resp. $\rho_*( \nu )$) is absolutely continuous with respect to $\mu _1,$ (resp. $\nu _1$)  
and  $E=(\theta \times \rho )^{-1}(E_1),$ where $E_1$ is a Borel subset of $X_1\times Y_1.$ If a nonzero compact operator 
(or a nonzero finite rank operator or a rank 1 operator) 
is supported on $E$, then there exists a nonzero 
compact operator (or a nonzero finite rank operator or a rank 1 operator)  supported on $E_1.$ 
Also if $\frak{M}_{min}(E)$ contains  a nonzero 
compact operator (or a nonzero finite rank operator or a rank 1 operator), then so does 
$\frak{M}_{min}(E_1).$ \end{theorem}

\bigskip

\em{Aknowledgement:} I am indebted to Dr.~I.~Todorov for helpful discussions on the topic 
of the paper. I wish to thank the anonymous referee for a number of 
useful suggestions which helped me improve the manuscript.

\end{document}